\documentclass{svmult}

\usepackage{type1cm}        
\usepackage{makeidx}         
\usepackage{graphicx}        
\usepackage{multicol}        
\usepackage[bottom]{footmisc}
\usepackage{bm}
\usepackage{subcaption}
\makeatletter
\newcommand{\Rmnum}[1]{\mathrm{\@Roman{#1}}}
\makeatother
\usepackage{xcolor}
\usepackage{soul}
\usepackage{cancel}
\usepackage{easyReview}
\usepackage{newtxtext}       
\usepackage{circuitikz}
\usepackage{tikz}%
\usepackage[varvw]{newtxmath}       
\usepackage{xspace}
\usetikzlibrary{patterns}
\usepackage{pgfplots}
\pgfplotsset{compat=1.18}
\usepackage{pgf}

\usepackage{ulem}
\usepackage{soul} 
\newcommand{\hcancel}[1]{\ifmmode\text{\st{#1}}\else\st{#1}\fi}

\newcommand{\bmA}{\ensuremath{\mathbf{A}}}
\newcommand{\bmB}{\ensuremath{\mathbf{B}}}

\newcommand{\bmD}{\ensuremath{\mathbf{D}}}

\newcommand{\bmH}{\ensuremath{\mathbf{H}}}

\newcommand{\bmJ}{\ensuremath{\mathbf{J}}}

\newcommand{\bmR}{\ensuremath{\mathbf{R}}}

\newcommand{\bmV}{\ensuremath{\mathbf{V}}}

\newcommand{\bmX}{\ensuremath{\mathbf{X}}}

\newcommand{\bmr}{\ensuremath{\mathbf{r}}}

\newcommand{\bmu}{\ensuremath{\mathbf{u}}}

\newcommand{\bmx}{\ensuremath{\mathbf{x}}}
\newcommand{\bmy}{\ensuremath{\mathbf{y}}}

\newcommand{\R}{\mathbb{R}}

\newcommand{\ConFuncSet}{\bm{C}}

\newcommand{\Jr}{\check{\bmJ}}
\newcommand{\Rr}{\check{\bmR}}

\newcommand{\Br}{\check{\bmB}}

\newcommand{\statex}{{\bmx}}
\newcommand{\stateapprox}{\tilde{\statex}}
\newcommand{\stateRed}{\check{\statex}}

\newcommand{\outputy}{{\bmy}}

\newcommand{\outputRed}{\check{\outputy}}
\newcommand{\inputu}{\bmu}

\newcommand{\statexsnap}{{\bmX}}

\newcommand{\Hamlabel}{\mathcal{H}}
\newcommand{\HamlabelRed}{\check{\mathcal{H}}}

\newcommand{\dHmat}{\bmH}
\newcommand{\dHmatRed}{\reduce{\dHmat}}

\newcommand{\ddt}{\frac{\mathrm{d}}{{\mathrm{d}} t}}
\newcommand{\ide}{\mathbb{I}}
\newcommand{\bzero}{\ensuremath{\mathbf{0}}} 

\newcommand{\snapsize}{n_t}

\newcommand{\portsize}{m}
\newcommand{\FOMsize}{N}
\newcommand{\ROMsize}{r}
\newcommand{\FOMsizeHalf}{n}

\newcommand{\residual}{{\bmr}}

\newcommand{\funcW}{\mathcal{W}}
\newcommand{\funcJ}{\bmJ}
\newcommand{\funcR}{\bmR}

\newcommand{\approxfunc}{\varphi}

\newcommand{\reducfunc}{\mathcal{R}}



\newcommand{\pH}{\textsf{pH}\xspace}
\newcommand{\GMG}{\textsf{GMG}\xspace}
\newcommand{\DEIM}{\textsf{DEIM}\xspace}

\newcommand{\POD}{\textsf{POD}\xspace}
\newcommand{\FOM}{\textsf{FOM}\xspace}
\newcommand{\ROM}{\textsf{ROM}\xspace}
\newcommand{\FOMs}{\textsf{FOM}s\xspace}
\newcommand{\ROMs}{\textsf{ROM}s\xspace}
\newcommand{\MOR}{\textsf{MOR}\xspace}

\newcommand{\NN}{\textsf{NN}\xspace}

\newcommand\LineWidth{0.8pt}

\newcommand{\bit}{\begin{itemize}}
	\newcommand{\eit}{\end{itemize}}
\newcommand{\ben}{\begin{enumerate}}
	\newcommand{\een}{\end{enumerate}}

\newcommand{\spann}[1]{\ensuremath{\mathop{\mathrm{span}}\left( #1 \right)}}

\newcommand{\reduce}[1]{\check{#1}}

\newcommand{\MarkSize}{3pt}

\newcommand{\errorstate}{\mathrm{error}_{x,\mathrm{red}}}
\newcommand{\erroroutput}{\mathrm{error}_{y}}

\DeclareMathOperator*{\argmin}{arg\,min}%

\makeindex             

\usepackage{lineno}
\usepackage{soul}

\newcommand{\oldtext}[1]{\textcolor{red}{}}

\newcommand{\oldtextone}[1]{{\textcolor{black}{}}}

\newcommand{\mycancel}[1]{}
\setreviewsoff

\begin{document}
	
	\title*{Model reduction of port-Hamiltonian systems via neural networks} 
	\author{Silke Glas\orcidID{0000-0003-3274-1615} and\\ 
		Alexander Heinlein\orcidID{0000-0003-1578-8104} and\\
		Harald Monsuur\orcidID{0000-0001-9482-2010} and\\
		Hongliang Mu\orcidID{0009-0003-2365-2829}}
	\institute{Silke Glas, Harald Monsuur and Hongliang Mu \at Department of Applied Mathematics, University of Twente, Enschede, The Netherlands, \email{s.m.glas@utwente.nl; harald.monsuur@utwente.nl;h.l.mu@utwente.nl}
		\\
		Alexander Heinlein \at Delft Institute of Applied Mathematics, Delft University of Technology, Delft, The Netherlands, \email{a.heinlein@tudelft.nl}}
	%
	%
	\maketitle
	
	\abstract{In this paper, we consider structure-preserving model reduction of port-Hamiltonian (\pH) systems which extend classical Hamiltonian systems with dissipation and an input-output port. 
	These \pH systems are often used in multi-physics systems, as the interconnection of one or more \pH systems results again in a \pH system. 
	If particularly the system matrices associated with the interconnection and/or dissipation of a pH system are state-dependent, then the evaluation of standard reduced-order models (\ROMs) may depend on the dimension of the original full-order model, resulting in high computational costs. To circumvent these high costs, we propose to use structure-preserving neural networks.
	In particular, we perform two steps: (1) we use the generalized manifold Galerkin projection from \cite{buchfink2024model} to project the pH system onto the reduced space; then (2) we train a neural network to learn the map from the reduced-order state to the reduced-order interconnection and dissipation system matrices. To ensure that the resulting \ROM is again a \pH system, the architecture of the neural network is chosen such that the skew-symmetry and positive semi-definiteness of the reduced-order systems matrices are maintained.
	In a numerical example, we consider a nonlinear mass-spring-damper system with state-dependent system matrices. The numerical results show that the proposed method achieves a significant computational speed-up compared to the original \ROM with comparable accuracy.
	}
		  
	\section{Introduction}
	Energy-based modeling in, e.g., electrical circuits \cite{falaize2016passive}, mechanical systems \cite{Duindam2009}, and fluid dynamical systems \cite{califano2021geometric}, often results in a port-Hamiltonian (\pH) form. In this paper, we consider \pH systems of the form 
	\begin{subequations}\label{equ:Nonlinear_FOM_pH}
		\begin{align}
				\ddt \statex(t) &=\left(\funcJ(\statex(t))-\funcR(\statex(t))\right)
				\nabla_{\statex}\Hamlabel(\statex(t))
				+\bmB \inputu(t), \label{equ:Nonlinear_FOM_pH_a}\\
				\outputy(t) &=\bmB^{\top}\nabla_{\statex}\Hamlabel(\statex(t)),
				\quad
				\statex(0) = \statex_0,\label{equ:Nonlinear_FOM_pH_b}
		\end{align}
	\end{subequations}
		where  $\statex\colon[0,\infty)\rightarrow\R^{\FOMsize}$ is the state of the system with its initial value given by $\statex_0\in\R^{\FOMsize}$. The matrix-valued functions $\funcJ,\funcR\colon\R^{\FOMsize}\rightarrow\R^{\FOMsize\times \FOMsize}$
		 model the interconnection structure and dissipation of the system and are skew-symmetric and positive semi-definite, i.e., 
		\begin{equation*}
			\funcJ(\statex) = -\funcJ^{\top}(\statex) \in \R^{\FOMsize\times \FOMsize},\
			0\preceq \funcR(\statex) = \funcR^{\top}(\statex) \in \R^{\FOMsize\times \FOMsize},\quad \
			\forall \statex\in\R^{\FOMsize}.
		\end{equation*} 
		The above \pH system \eqref{equ:Nonlinear_FOM_pH} exchanges energy with the environment via the $\portsize$-dimensional ports described by $\bmB\in\R^{\FOMsize\times \portsize}$, which is often denoted as port matrix. 
		The function $\Hamlabel: \R^{\FOMsize}\rightarrow \R$, which needs to satisfy
		\begin{equation}
			\label{equ:Ham-condition}
			\exists \statex_e\in\R^{\FOMsize},\ 
			\textrm{s.t.}\
			\nabla_{\statex}\Hamlabel(\statex_e) = \bzero,\
			\textrm{and}\
			\Hamlabel(\statex)\geq\Hamlabel(\statex_e)\quad \forall \statex\in\R^{\FOMsize},
		\end{equation}
		is called the Hamiltonian and represents the internal energy of the system. 
		Finally, the functions $\inputu,\outputy:[0,\infty)\rightarrow\R^{\portsize}$ are the input and output of the system. 
		Moreover, we assume that
		\begin{equation*}
			{\rm det}(\funcJ(\statex)-\funcR(\statex))\neq 0,\quad
			\forall \statex\in\R^{\FOMsize}.
		\end{equation*}
			One property of \pH systems is, that they satisfy a power 
	balance equation \cite{vanderschaft2014port}, i.e., 
		\begin{equation*}
			\ddt\Hamlabel(\statex(t)) = \outputy(t)^{\top}\bmu(t)-
			\left(\nabla_{\statex}\Hamlabel(\statex(t))\right)^{\top}
			\funcR(\statex(t))
			\nabla_{\statex}\Hamlabel(\statex(t)).
		\end{equation*}
		The physical meaning of this power balance equation is that the change of the Hamiltonian of a pH system equals the power supplied to the system by the input-output ports minus the dissipated power of the system.
		As a result, a \pH system is passive, i.e., it does not generate energy itself.  For further details about \pH systems, we refer the reader to, e.g., \cite{chaturantabut2016structure,polyuga2012effort,vanderschaft2014port,schulze2023structure}. 

If the state dimension $\FOMsize$ of the \pH system \eqref{equ:Nonlinear_FOM_pH} is high, then a projection-based model order reduction (\MOR) method can be applied to construct a reduced-order \pH system, which has lower simulation cost. However, without taking into account the \pH structure in the reduction process, this \pH structure is not necessarily kept in the \ROM. In the following, we recall some works, where the \pH structure in the \ROM can indeed be preserved. 

A moment matching-based structure-preserving \MOR method for linear \pH system has been proposed in \cite{gugercin2012structure}.
A similar projection framework is used in \cite{schulze2023structure}, where a \MOR method for nonlinear \pH systems that uses a separable nonlinear approximation ansatz is introduced.
The authors of \cite{chaturantabut2016structure} proposed a  structure-preserving \MOR method for nonlinear \pH systems. 
This method uses linear \MOR techniques and presents a structure-preserving discrete empirical interpolation method (\DEIM) method for the nonlinear Hamiltonian.
Furthermore, a \MOR method for \pH differential-algebraic equations is introduced in \cite{schwerdtner2022structure}, and for a more general overview of \pH differential-algebraic equations including model reduction, we refer to \cite{mehrmann2023control}.
The generalized manifold Galerkin (\GMG) projection is discussed by the authors of \cite{buchfink2024model}, where as a special instance structure-preserving \MOR methods for Hamiltonian systems and Lagrangian systems are considered. Recently, the authors of \cite{GM2026structure} extended the \GMG projection to port-Hamiltonian systems. This method can use a general nonlinear approximation map and can directly be applied to nonlinear \pH systems of the form  \eqref{equ:Nonlinear_FOM_pH}.
 A nonintrusive \MOR method for \pH systems has been proposed in \cite{rettberg2025data}, where a neural network is used to learn a linear reduced-order \pH system. 

In this paper, we consider the case of state-dependent interconnection and dissipation matrices $\funcJ(\statex)$, $\funcR(\statex)$, where the dependence on the state could be even nonlinear. This means that when using projection-based \MOR methods to derive a reduced system, the reduced interconnection and dissipation matrices would still be state-dependent and thus depend on the dimension of the full-order model \FOM. One possible way to circumvent this, would be the matrix DEIM introduced in \cite{bonomi2017matrix}. However, the \pH structure is not guaranteed to be preserved by applying this algorithm. 
In this paper, we overcome this challenge by using a neural network to approximate the state-dependent system matrices. In more detail, the proposed algorithm is a two-step method: (i) a \ROM is built by using the \GMG projection; (ii) we train a \NN whose input is given by the reduced-order trajectory and whose outputs are the system matrices. This results in an approximated reduced-order \pH system whose interconnection and dissipation matrices are constructed by a \NN.
We use a \NN that has a similar structure as the \NN used in \cite{rettberg2025data} such that the \pH structure is preserved.
Different to \cite{rettberg2025data},
we focus on nonlinear reduced-order \pH systems where $\funcJ$ and $\funcR$ are state-dependent.

The outline of this paper is as follows. In Section \ref{Sec:GMG_proj}, we introduce some background of projection-based model reduction of \pH systems and the \GMG reduction. Then, the proposed structure-preserving \MOR method based on neural networks is introduced in Section \ref{Sec:GMG_NN}. Subsequently, we present numerical results of a nonlinear mass-spring-damper system in Section \ref{Sec:Num_examp}. We conclude and discuss ideas for further work in Section \ref{Sec:Conclusion_Outlook}.
		
	\section{Projection-based \MOR for \pH systems and the Generalized manifold Galerkin reduction}
	\label{Sec:GMG_proj}
	
A general framework of projection-based \MOR methods of \pH systems can be derived as follows. First, an approximation map $\approxfunc\in\ConFuncSet^1(\R^{\ROMsize},\R^{\FOMsize})$ of the state, with $\ROMsize\ll\FOMsize$, is chosen such that 
\begin{equation}
	\label{equ:state_approximation} 
	\stateapprox(t) \coloneq\approxfunc(\stateRed(t)) \approx\statex(t), 
\end{equation}
where $\stateRed(t) \in \R^{r}$ are the reduced coordinates and $\stateapprox(t) \in \R^{N}$ is an approximated solution, also known as the reconstructed solution. 
Due to the chain rule, the evolution in time of the reconstructed solution $\stateapprox$ is
$\ddt\stateapprox(t)=\bmD_{\stateRed}\approxfunc(\stateRed)\ddt\stateRed(t)$, where $\bmD_{\stateRed}\approxfunc(\stateRed)$ denotes the Jacobian of $\approxfunc$. For readability, from now on, we drop the dependence on $\stateRed$ and denote this Jacobian by $\bmD\approxfunc$.
Then, we define the time-continuous residual of the \FOM \eqref{equ:Nonlinear_FOM_pH} with respect to the approximation $\stateapprox(t)$ by 
\begin{align*}
	\residual(t)
	&\coloneq
	\ddt \stateapprox(t)
	-
	\left(\funcJ( \stateapprox(t))-\funcR( \stateapprox(t))\right)
	\nabla_{\stateapprox}\Hamlabel( \stateapprox(t))
	-
	\bmB \inputu(t) \\
	&= 
	\bmD_{\stateRed}\approxfunc(\stateRed)\ddt\stateRed(t)
	-
	\left(\funcJ( \approxfunc(\stateRed(t)))-\funcR(\approxfunc(\stateRed(t)))\right)
	\nabla_{\stateapprox}\Hamlabel( \approxfunc(\stateRed(t)))
	-
	\bmB \inputu(t). 
\end{align*}
Further, we assume 
	that 
\begin{equation}
	\label{equ:non_singular_reduce_JR}
	\det\left(
	\bmD\approxfunc^{\top}
	(\funcJ(\approxfunc(\stateRed))-\funcR(\approxfunc(\stateRed))
	\bmD\approxfunc
	\right)
	\neq
	0,\quad
	\forall
	\stateRed\in\R^{\ROMsize}.
\end{equation}
Based on the idea of the generalized manifold Galerkin (\GMG) reduction \cite{buchfink2024model}, we construct a reduction map $\reducfunc_{\stateRed}\colon \statex\mapsto \funcW(\stateRed)^{\top}\statex$ with $\funcW\colon\R^{\ROMsize}\mapsto\R^{\FOMsize\times\ROMsize}$ given by
\begin{equation}
	\label{equ:GMG_proj_basis}
	\funcW\colon 
	\stateRed\mapsto
	(\funcJ(\approxfunc(\stateRed))-\funcR(\approxfunc(\stateRed)))^{-\top}
	\bmD\approxfunc
	(\bmD\approxfunc^{\top}(\funcJ(\approxfunc(\stateRed))-\funcR(\approxfunc(\stateRed))^{-\top} \bmD\approxfunc)^{-1}.
\end{equation} 
Subsequently, we derive a \ROM by assuming, that the projection of the time-continuous residual vanishes, i.e., 
\begin{equation}
	\label{equ:residual_vanish}
	\funcW(\stateRed(t))^{\top}\residual(t)\overset{!}{=}0.
\end{equation}
By reformulating \eqref{equ:residual_vanish}, we arrive at
\begin{equation}
	\label{equ:ROM_dif_part}
	\begin{array}{rcl}
		\ddt \stateRed(t)
		=
		\left(\reduce{\funcJ}( \stateRed(t))-
		\reduce{\funcR}( \stateRed(t))\right)
		\bmD\approxfunc^{\top}
		\nabla_{\stateapprox(t)}\Hamlabel(\stateapprox(t))
		-
		\funcW(\stateRed(t))^{\top}	\bmB \inputu(t),
	\end{array}
\end{equation}
where 
$\reduce{\funcJ}\colon\stateRed\mapsto
\funcW(\stateRed)^{\top}\funcJ(\approxfunc(\stateRed))\funcW(\stateRed),
\
\reduce{\funcR}\colon\stateRed\mapsto
\funcW(\stateRed)^{\top}\funcR(\approxfunc(\stateRed))\funcW(\stateRed).$
Meanwhile, the output of the system is given by 
\begin{equation}
	\label{equ:ROM_output_part}
	\outputRed(t)\coloneq\bmB^{\top}\nabla_{\approxfunc(\stateRed)}\Hamlabel(\approxfunc(\stateapprox(t))).
\end{equation}
In this work, we consider linear approximation maps only, i.e., $\approxfunc\colon\stateRed\mapsto \bmV \stateRed $, where $\bmV\in\R^{\FOMsize\times\ROMsize}$. The following theorem shows how the \pH structure is preserved.
\begin{theorem}[Reduced-order \pH systems with \GMG reduction (see Theorem 3.1 in \cite{GM2026structure}.]
	\label{theo:GMG_state_dependent_JR}
	Consider a \pH system \eqref{equ:Nonlinear_FOM_pH}. 
	Let $\approxfunc\colon\stateRed\mapsto \bmV \stateRed $, with $\bmV\in\R^{\FOMsize\times\ROMsize}$ be an approximation map such that $\spann{\bmB}\subseteqq\spann{\bmV}$. Further, we consider the reduced initial condition $\stateRed(0)=\bmV^{\dagger}\statex_0$, where the notation $(\cdot)^{\dagger}$ represents the Moore-Penrose pseudoinverse.
	If the assumption given in \eqref{equ:non_singular_reduce_JR} holds,
	then \eqref{equ:ROM_dif_part} together with \eqref{equ:ROM_output_part} yields a (reduced) \pH system, with initial value $\stateRed(0)$.	

\end{theorem}
\begin{proof}
	Although, we consider \pH systems with explicit state-dependent $\funcJ$ and $\funcR$, the proof of this theorem is analogous to that of Theorem 3.1 in \cite{GM2026structure}.
\end{proof}

In this work, we would like to build the approximation map $\approxfunc$ by a data-driven method.
Given a high-dimensional \pH system \eqref{equ:Nonlinear_FOM_pH} and a time discretization $\left\{t_0,t_1,\ldots,t_{\snapsize}\right\}$, we denote the trajectory of this system by 
\begin{equation*}
	\bmX = 
	\begin{pmatrix}
		\statex(t_0),\statex(t_1),\ldots,\statex(t_{\snapsize})
	\end{pmatrix}
	\in\R^{\FOMsize\times(\snapsize+1)}.
\end{equation*}
Then, for a given reduced-order $\ROMsize$ with  $\portsize < \ROMsize \ll \FOMsize)$, 
a linear approximation map $\approxfunc\colon\stateRed\mapsto \bmV \stateRed $ is built, where $\bmV\in\R^{\FOMsize\times\ROMsize}$ is constructed by
\begin{equation}
	\label{equ:GMG_approx_basis}
	\bmV=
	\begin{pmatrix}
		\bmB,\bmV_{\rm POD}
	\end{pmatrix},
	\  
	\bmV_{\mathrm{ POD}}
	=\argmin_{\bmA\in\R^{\FOMsize\times(\ROMsize-\portsize)}, \bmA^{\top}\bmA=\ide_{\ROMsize-\portsize}}
	\|
	(\ide_{\FOMsize}-\bmA\bmA^{\top})(\bmX-\bmB\bmB^{\dagger}\bmX)
	\|_F.
\end{equation}
	The conditions for $\approxfunc$ of Theorem \ref{theo:GMG_state_dependent_JR} are satisfied by the above construction. 
	In this work, we denote the resulting reduced-order \pH system by \GMG-\POD-\ROM.

	\section{Neural network-learned reduced-order pH systems}
	\label{Sec:GMG_NN}
	
	Our ultimate goal is to design a \ROM that preserves the \pH structure and reduces the computational cost of simulating the original \FOM. For this matter, it is crucial that the evaluation of the reduced-order system does not depend on the dimension of the \FOM. For \pH problems with constant $\funcJ, \funcR$ and a quadratic Hamiltonian, this is easily achieved as the matrices arising in the reduced \pH system are of small dimension and constant and can hence be pre-computed.

	The situation becomes more delicate for the state-dependent problem we consider here. 
	The reduced port matrix $\Br$ remains constant and possible nonlinear reduced Hamiltonian $\HamlabelRed$ could be evaluated independent of the dimension of the \FOM, for example, by using a structure-preserving \DEIM. However, the evaluation of the reduced interconnection and dissipation matrices $\reduce{\funcJ}$ and $\reduce{\funcR}$ still depends on the dimension of the \FOM. In fact, from 
	\begin{equation*}
		\begin{array}{rl}
			  (\reduce{\funcJ}(\stateRed)-\reduce{\funcR}(\stateRed))
			=&\funcW(\stateRed)^{\top}
			     \left(
			     \funcJ(\bmV\stateRed)-\funcR(\bmV\stateRed)
			     \right)
			     \funcW(\stateRed)\\
			 =&
			   \left(
			   \bmV^{\top}
			   \left(
			   \funcJ(\bmV\stateRed)-\funcR(\bmV\stateRed)
			   \right)^{-1}
			   \bmV
			   \right)^{-1},
		\end{array}
	\end{equation*}
	we see that the evaluation of $\reduce{\funcJ}(\stateRed)-\reduce{\funcR}(\stateRed)$ requires the inversion of a $\FOMsize\times\FOMsize$ matrix, which can be very expensive.
	In principle, one could approximate $\Jr$ and $\Rr$ using some method of interpolation. However, ensuring that the approximation to $\Rr$ is positive semi-definite can be quite challenging.

	Inspired by \cite{rettberg2025data}, we tackle this challenge by defining a neural network that in turn defines a class of skew-symmetric and positive semi-definite matrices. To be precise, we train a network $\Theta_{JR}\colon\R^{\ROMsize}\rightarrow\R^{\ROMsize\times \ROMsize}$, which, by considering the the lower triangular and strictly upper triangular parts of its output, defines two mappings $\Theta_R,\Theta_J\colon\R^{\ROMsize}\rightarrow\R^{\ROMsize\times \ROMsize}$. 
	Then we define $\Phi_J$ and $\Phi_R$ by
	\begin{equation}
		\label{equ:NN_sturcture}
		\Phi_J\colon \stateRed\mapsto \Theta_J(\stateRed)-\Theta_J(\stateRed)^{\top},\quad
		\Phi_R\colon \stateRed\mapsto \Theta_R(\stateRed)\Theta_R(\stateRed)^{\top}.
	\end{equation}
	By construction, we have that $\Phi_J$ is skew-symmetric and $\Phi_R$ is symmetric positive semi-definite, i.e.,
	\begin{equation*}
		\Phi_J(\stateRed)=-\Phi_J(\stateRed)^{\top},
		\quad
		\Phi_R(\stateRed)=\Phi_R(\stateRed)^{\top}\succeq 0,
		\quad
		\forall
			\stateRed\in\R^{\ROMsize}.
	\end{equation*}
	As a result, this neural network allows the approximation of $\Jr$ and $\Rr$ by $\Phi_J$ and $\Phi_R$ respectively, while maintaining the \pH structure in the \ROM, which is denoted by \GMG-\POD-\NN-\ROM. 

	To train the network $\Theta_{JR}$, we need to arrange its input and output data. With a snapshot matrix $\statexsnap\in\R^{\FOMsize\times(\snapsize+1)}$ and a projection basis $\bmV$, we construct the reduced snapshot matrix by $\reduce{\statexsnap}\coloneq\bmV^{\dagger}\statexsnap\in\R^{\ROMsize\times(\snapsize+1)}$.
	For $k =1,\ldots, \snapsize+1$, we denote $\stateRed_{k}$ as the $k-$th column of $\reduce{\statexsnap}$ and
	\begin{equation*}
		\reduce{\bmJ}_k:=\reduce{\funcJ}(\stateRed_{k}),
		\quad
		\reduce{\bmR}_k:=\reduce{\funcR}(\stateRed_{k}),
		\quad
		\bmD\dHmatRed_k:=\bmV^{\top}\nabla_{\statex}\Hamlabel(\bmV \stateRed_{k}).
	\end{equation*}

	To improve numerical stability and performance of the neural network, we apply standardization \cite{shanker1996effect} to the input and output data. However, if a standardization is applied to $\Phi_J$ and $\Phi_R$ directly, then the structure of the \pH system is destroyed. Thus, we apply a standardization on $\Theta_J$ and $\Theta_R$, which is applied in the following way. 
	
	First, we denote
	\begin{equation*}
		\label{equ:pre_normalization_mats}
		\reduce{\bmJ}_{\mathrm{triu}, k}
		\coloneq
		\mathrm{triu}
		(\reduce{\bmJ}_k),\quad
		\reduce{\bmR}_{\mathrm{chol}, k}
		\coloneq
		\mathrm{chol}
		(\reduce{\bmR}_k+10^{-14}\ide_{\ROMsize})
		\footnote{$\reduce{\bmR}_k$ can be a singular matrix. Thus, we compute the Cholesky decomposition of $(\reduce{\bmR}_k+10^{-14}\ide_{\ROMsize})$ instead.}
		,\quad
		k=1,\ldots,\snapsize+1,
	\end{equation*} 
	where $``\mathrm{triu}"$ and $``\mathrm{chol}"$ are two operators that return square matrices which is the strictly upper-triangular part and a matrix of the Cholesky decomposition of a matrix respectively.  
	Then we build $\reduce{\bmJ}_{\mathrm{triu}, \mathrm{mean}}, \reduce{\bmJ}_{\mathrm{triu}, \mathrm{std}}$ as the element-wise mean value and standard variance of $\{\reduce{\bmJ}_{\mathrm{triu},k}\}_{k=1}^{\snapsize+1}$. Similarly, we build $\reduce{\bmR}_{\mathrm{chol},  \mathrm{mean}}, \reduce{\bmR}_{\mathrm{chol},\mathrm{std}}$ based on $\{\reduce{\bmR}_{\mathrm{chol},k}\}_{k=1}^{\snapsize+1}$.
	With these matrices, we define $\widetilde{\Theta}_J$ and $\widetilde{\Theta}_R$ by
	\begin{equation*}
		\label{equ:def_normalization_Theta}
		\widetilde{\Theta}_J \colon
		\stateRed \mapsto
		\reduce{\bmJ}_{\mathrm{triu}, \mathrm{mean}}
		+
		\reduce{\bmJ}_{\mathrm{triu}, \mathrm{std}}*\Theta_J(\stateRed),
		\quad
		\widetilde{\Theta}_R \colon
		\stateRed \mapsto
		\reduce{\bmR}_{\mathrm{chol},  \mathrm{mean}}
		+
		\reduce{\bmR}_{\mathrm{chol}, \mathrm{std}}
		*\Theta_R(\stateRed),
	\end{equation*}
	where ``*" represents element-wise multiplication.
	After that, we replace $\Theta_J$ and $\Theta_R$ in \eqref{equ:NN_sturcture} by $\widetilde{\Theta}_J$ and $\widetilde{\Theta}_R$. 

	Finally, we train the network $\Theta_{JR}$ with a loss function defined as follows:
	\begin{equation}
		\label{equ:loss_func}
		L(\theta) = \lambda L_{\bmJ,\bmR}(\theta) + (1-\lambda) L_{\mathrm{rhs}}(\theta),
	\end{equation}
	where $\lambda\in(0,1)$,
	 $\theta$ represents the hyper-parameters of $\Theta_{JR}$, and $L_{\bmJ,\bmR}(\theta)$ and $L_{\mathrm{rhs}}(\theta)$ are two loss terms defined by
	\begin{align*}
		L_{\bmJ,\bmR}(\theta) &= \frac{1}{\ROMsize^2|\chi|}\sum_{k\in\chi}\|(\Phi_J(\stateRed_{k})-\Phi_R(\stateRed_{k})) - (\Jr_k-\Rr_k))\|_F^2,
		\\
		L_{\mathrm{rhs}}(\theta) &=  \frac{1}{\ROMsize|\chi|}\sum_{k\in\chi}				\|\left((\Phi_J(\stateRed_{k})-\Phi_R(\stateRed_{k})) - (\Jr_k-\Rr_k))\right) \bmD \dHmatRed_k\|_2^2,
	\end{align*}
	where $\chi$ represents the training set or the test set.
	The loss function $L_{\bmJ,\bmR}$ ensures that $(\reduce{\funcJ}-\reduce{\funcR})$ can be approximated closely by $(\Phi_J-\Phi_R)$, and the dynamics of the system are approximated based on $L_{\mathrm{rhs}}$.

	\section{Numerical example}
	\label{Sec:Num_examp}

	In this section, we consider a nonlinear mass-spring-damper system with state-dependent dissipation matrix $\funcR$, 
	which is similar to the example in \cite{kawano2018structure}.
	This system is depicted in Figure \ref{fig:NLmsd_sys}, where $\xi_i \in \R$ denotes the position of the mass $m_i \in \R^{+}$.
	\begin{figure}[h]
		\centering
		\begin{circuitikz}[scale = 0.65][H]
			\def\masshalfheight{0.75}   
			\def\wallhalfheight{1.0} 
			\def\wallwidth{0.25} 
			\def\masswidth{1.5}    
			\def\dampwidth{3.0}
			\def\cdotswidth{2.0}
			\def\damppos{-0.5}
			\def\springpos{0.5}
			\def\springscale{2}
			\def\damplabdis{5mm}
			\def\springlabdis{1mm}
			\tikzset{every node/.style={font=\normalsize}}
			\tikzset{every path/.style={line width=1pt}}
			
			\draw[->] (-1.5,0) -- (0,0);
			\node at (-0.5,0.5) {$u(t)$};
			\draw[fill=gray!40] (0, -\masshalfheight) rectangle (\masswidth,\masshalfheight);
			\node at (0.5*\masswidth, 0) {$m_{1}$};
			
			\draw (\masswidth,\springpos) to[spring, l=$k_1+k_2(\xi_{2}-\xi_{1})^2$, label distance=\springlabdis] (\masswidth+\dampwidth, \springpos);
			\draw (\masswidth,\damppos) to[damper, l_=$c_1+c_2(\dot{\xi}_{2}-\dot{\xi}_1)^2$, label distance=\damplabdis] (\masswidth+\dampwidth, \damppos);
			
			\draw[fill=gray!40] (\masswidth+\dampwidth, -\masshalfheight) rectangle (2*\masswidth+\dampwidth,\masshalfheight);
			\node at (1.5*\masswidth+\dampwidth, 0) {$m_{2}$};
			
			\draw (2*\masswidth+\dampwidth,\springpos) to[spring, l=$k_1+k_2(\xi_{3}-\xi_{2})^2$, label distance=\springlabdis] (2*\masswidth+2*\dampwidth, \springpos);
			\draw (2*\masswidth+\dampwidth,\damppos) to[damper, l_=$c_1+c_2(\dot{\xi}_{3}-\dot{\xi}_2)^2$, label distance=\damplabdis] (2*\masswidth+2*\dampwidth, \damppos);
			
			\node at (2*\masswidth+2*\dampwidth+0.5*\cdotswidth, 0) {\Large$\cdots$};
			
			\draw[fill=gray!40] (2*\masswidth+2*\dampwidth+\cdotswidth, -\masshalfheight) rectangle (3*\masswidth+2*\dampwidth+\cdotswidth,\masshalfheight);
			\node at (2.5*\masswidth+2*\dampwidth+\cdotswidth, 0) {$m_{\FOMsizeHalf}$};
			
			\draw (3*\masswidth+2*\dampwidth+\cdotswidth,\springpos) to[spring, l=$k_1$, label distance=\springlabdis] (3*\masswidth+3*\dampwidth+\cdotswidth, \springpos);
			\draw (3*\masswidth+2*\dampwidth+\cdotswidth,\damppos) to[damper, l_=$c_1$, label distance=\damplabdis] (3*\masswidth+3*\dampwidth+\cdotswidth, \damppos);
			
			\pattern[pattern=north east lines] (3*\masswidth+3*\dampwidth+\cdotswidth,-\wallhalfheight) rectangle (3*\masswidth+3*\dampwidth+\cdotswidth+\wallwidth,\wallhalfheight);
			\draw[thick] (3*\masswidth+3*\dampwidth+\cdotswidth,-\wallhalfheight) -- (3*\masswidth+3*\dampwidth+\cdotswidth,\wallhalfheight);
		\end{circuitikz}
		\caption{Illustration of mass-spring-damper system with nonlinear spring and damper.}
		\label{fig:NLmsd_sys}
	\end{figure}
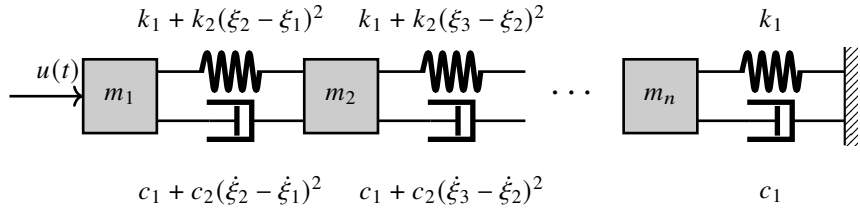
	For $i=1,\ldots,\FOMsizeHalf-1$,  the masses $m_i $ and $m_{i+1}$ are connected by
	a nonlinear spring and a nonlinear damper with spring and damper coefficients given by
	\begin{equation*}
		k_1+k_2(\xi_{i+1}-\xi_{i})^2,\
		c_1+c_2(\dot{\xi}_{i+1}-\dot{\xi}_{i})^2,
	\end{equation*}
	where $k_1,k_2,c_1,c_2\geq 0.$
	The last mass $m_{\FOMsizeHalf}$ is connected to a wall with a linear spring and damper, with spring and damper coefficients given by $k_1$ and $c_1$, respectively. 
	The internal energy of this system consists of the kinetic and potential energy. 
	
	Now, we set $\statex=
	\begin{pmatrix}
		\xi,\ldots,\xi_{\FOMsizeHalf},\dot{\xi}_1,\ldots,\dot{\xi}_{\FOMsizeHalf}
	\end{pmatrix}^{\top}$
	to be the system state and arrive at the following \pH system
	\begin{align*}
		\ddt \statex(t) &= (\bmJ-\funcR(\statex(t)))\nabla_{\statex}\Hamlabel(\statex(t))+\bmB\bmu(t),\\
		\outputy(t) &= \bmB^{\top}\nabla_{\statex}\Hamlabel(\statex(t)),
		\quad
		\statex(0)=\bzero_{2\FOMsizeHalf,1},
	\end{align*}
	where we have $\Hamlabel\colon\R^{2\FOMsizeHalf}\rightarrow\R$, and $\funcR\colon\R^{2\FOMsizeHalf}\rightarrow\R^{2\FOMsizeHalf\times 2\FOMsizeHalf}$ with 
	\begin{align*}
	   \Hamlabel(\statex)&= 
		\sum_{i=1}^{\FOMsizeHalf}\frac{1}{2}m_i\statex_{\FOMsizeHalf+i}^2
		+\sum_{i=1}^{\FOMsizeHalf-1}
		(\frac{1}{2}k_1 (\statex_{i+1}-\statex_{i})^{2} +\frac{1}{4}k_2 (\statex_{i+1}-\statex_{i})^4)
		+
		\frac{1}{2}k_1 \statex_{\FOMsizeHalf}^{2},\\
		\bmB &= \left[\begin{array}{c}
			\bzero_{\FOMsizeHalf,1}\\
			1\\
			\bzero_{\FOMsizeHalf-1,1}
		\end{array}\right],
		\quad
		\bmJ= \left[\begin{array}{cc}
			\bzero_{\FOMsizeHalf,\FOMsizeHalf} & \ide_{\FOMsizeHalf}\\
			-\ide_{\FOMsizeHalf} & \bzero_{\FOMsizeHalf,\FOMsizeHalf}
		\end{array}\right],
		\quad
		\funcR(\statex) = \left[\begin{array}{cc}
			\bzero_{\FOMsizeHalf,\FOMsizeHalf} & \bzero_{\FOMsizeHalf,\FOMsizeHalf}\\
			\bzero_{\FOMsizeHalf,\FOMsizeHalf} & \funcR_{22}(\statex)
		\end{array}\right].
	\end{align*}
	The matrix $\funcR_{22}(\statex) \in \R^{\FOMsizeHalf\times \FOMsizeHalf}$ is given by
		\begin{align*}
		\funcR_{22}(\statex)=
		\left[
		\begin{array}{ccccc}
			d_{1}(\statex)& -d_{1}(\statex) & 0 & \cdots &\cdots\\
			-d_{1}(\statex) & d_{1}(\statex) + d_{2}(\statex) & -d_{2}(\statex) & 0 & \cdots\\
			\vdots & \ddots &\ddots&\ddots & \vdots\\
			0&\cdots & -d_{\FOMsizeHalf-2}(\statex) & d_{\FOMsizeHalf-2}(\statex) + d_{\FOMsizeHalf-1}(\statex) & -d_{\FOMsizeHalf-1}(\statex)\\
			0&\cdots & 0 & -d_{\FOMsizeHalf-1}(\statex) & d_{\FOMsizeHalf-1}(\statex)
		\end{array}
		\right]
	\end{align*}
	with $d_{i}(\statex) = c_1 + c_2(\statex_{\FOMsizeHalf+i}-\statex_{\FOMsizeHalf+i})^2/4.$
	
	In this work, we use $\FOMsizeHalf=200$ resulting in a \FOM of dimension $ 2\FOMsizeHalf = 400$. For the spring and damper coefficients, we set $k_1 = k_2 = c_1 = c_2 =1$ and for the masses, we utilize $m_1=\ldots=m_{\FOMsizeHalf}=1$. The \FOM is simulated on a time interval of $(0,10]$. We use a uniform grid of the time interval with step size $\Delta t = 2\times 10^{-2}$, which results in the number of snapshots to be $\snapsize=500$. We use the $4$-th order Gauss-Legendre method for the time integration, which preserves the \pH structure \cite{kotyczka2018discrete}.
	Similarly as in \cite{kawano2018structure}, for the input, we use the sinusoidal input $u(t)=10*(\sin(t)+\sin(2t))$.
	For the settings of the neural network, we use the ReLU activation function and set LR = $3\times 10^{-3}$, WD $= 10^{-4}$, epochs$ = 10^3$, $\lambda=0.1$.
	\footnote{The hyper-parameters LR, WD, $\lambda_1$ and $\lambda_2$ are selected by a grid search. For the numerical example we use, we found that the performance of the \NN is not sensitive to these hyper-parameters.}
	
\textbf{Construction of \ROM:} For the computation of the \ROM, we first compute a linear approximation map $\approxfunc\colon\stateRed\mapsto\bmV\stateRed$, where $\bmV$ is computed as in \eqref{equ:GMG_approx_basis}. Then, 
    to ensure that we have a sufficiently large dataset to train the \NN, we build an extended snapshot matrix by 
    		\begin{equation*}
    				\reduce{\statexsnap}_{\rm ex} =
    				\begin{pmatrix}
    						\reduce{\statexsnap},
    						\reduce{\statexsnap}_{1},
    						\reduce{\statexsnap}_{2},
    						\reduce{\statexsnap}_{3},
    						\reduce{\statexsnap}_{\rm per}
    					\end{pmatrix},
    		\end{equation*}
    		where 
    		\begin{align*}
    				\reduce{\statexsnap}_{1} &= \frac{3}{4} \reduce{\statexsnap}[:,1:\snapsize] +\frac{1}{4} \reduce{\statexsnap}[:,2:\snapsize+1],\\
    				\reduce{\statexsnap}_{2} &= \frac{1}{2} \reduce{\statexsnap}[:,1:\snapsize] +\frac{1}{2} \reduce{\statexsnap}[:,2:\snapsize+1],\\
    				\reduce{\statexsnap}_{3} &= \frac{1}{4} \reduce{\statexsnap}[:,1:\snapsize] +\frac{3}{4} \reduce{\statexsnap}[:,2:\snapsize+1],\\
    				\reduce{\statexsnap}_{\rm per} &= 
    				\begin{pmatrix}
    						\reduce{\statexsnap},
    						\reduce{\statexsnap}_{1},
    						\reduce{\statexsnap}_{2},
    						\reduce{\statexsnap}_{3}
    					\end{pmatrix}
    				+ \epsilon\rm{randn}(\ROMsize, 4\snapsize + 1).
    			\end{align*}
    			The term $\reduce{\statexsnap}_{\rm per}$ is a perturbed input that improves the robustness of the \NN
    			and we set $\epsilon=2\times10^{-2}$.
    			The loss function is evaluated on the test set every 10 epochs, and the training stops if the test loss does not improve over 5 consecutive evaluations to avoid overfitting.

\textbf{The influence of the nonlinear damper:}
In the following, we investigate the effect of the nonlinear damping term by comparing the \FOM outputs for $c_2=1$ and $c_2=0$. The case $c_2=0$ corresponds to neglecting the nonlinear damping term. 
Figure \ref{fig:MSD_comp_FOM} depicts the behavior of the two  different above described \FOMs. The relative difference of these two \FOMs are
\begin{align*}
	\small
	&\mathrm{error}_{\hat{\statex}}
	=
	\sqrt{
		\frac{\sum_{i=0}^{\snapsize}
			\|\statex(t_i)-\hat{\statex}(t_i))\|_2^2
		}
		{
			\sum_{i=0}^{\snapsize}
			\|\statex(t_i)\|_2^2
		}
	}
	=
	9.51\times10^{-2}
	,&\\
	&\mathrm{error}_{\hat{\outputy}}
	=
	\sqrt{
		\frac{\sum_{i=0}^{\snapsize}
			\|\outputy(t_i)-\hat{\outputy}(t_i)\|_2^2
		}
		{
			\sum_{i=0}^{\snapsize}
			\|\outputy(t_i)\|_2^2
		}
	}
	=
	1.85\times10^{-1}
	,&
\end{align*}
where $\statex, \outputy$ are state and output of the \FOM with $c_2=1$, and $\hat{\statex}, \hat{\outputy}$ are state and output of the \FOM with $c_2=0$. 
This shows, that simply neglecting the nonlinear term yields insufficient results.

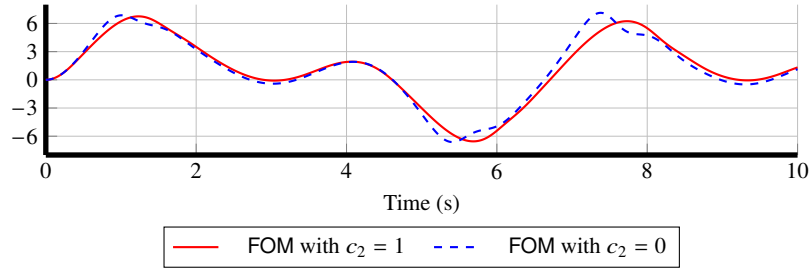
\begin{figure}[ht]
	\centering 
	\begin{tikzpicture}
		\begin{axis}[
			width=.85\linewidth,
			height=.17\linewidth, 
			scale only axis,
			grid=both,
			grid style={line width=.1pt,, draw=gray!10},
			major grid style={line width=.2pt,, draw=gray!50},
			axis lines*=left,
			axis line style={line width=2pt},
			no markers,
			xmin=0, xmax=10,
			ymin=-8, ymax=8,
			xtick={0,2,4,6,8,10},    
			ytick={-6,-3,0,3,6},
			legend columns=2,
			legend entries={\FOM with $c_2 = 1$, \FOM with $c_2 = 0$},
			legend style={
				font=\small,
				at={(.5,-0.75)},   
				anchor=south,
				column sep=1em
			},
			legend cell align={left},
			xlabel= Time (s)
			]
			\addplot [line width=\LineWidth, red]
			table[x index=0, y index=1, col sep=comma] {Y_FOM.dat};
			\addplot [line width=\LineWidth, blue, dashed]
			table[x index=0, y index=1, col sep=comma] {Y_FOM_const_R.dat};
		\end{axis}
	\end{tikzpicture}
	\caption{Comparison of the outputs $y(t)$ of \FOMs with different $c_2$.}
	\label{fig:MSD_comp_FOM}
\end{figure}

\textbf{Evaluation of the \ROM:} To evaluate the performance of the \ROMs, we compute both the relative reduction error with respect to the state and the relative output error defined by
		\begin{equation*}
			\small
			\errorstate
			=
			\sqrt{
				\frac{\sum_{i=0}^{\snapsize}
					\|\statex(t_i)-\approxfunc(\stateRed(t_i))\|_2^2
				}
				{
					\sum_{i=0}^{\snapsize}
					\|\statex(t_i)\|_2^2
				}
			},
			\quad
			\erroroutput
			= 
			\sqrt{
				\frac{\sum_{i=0}^{\snapsize}
					\|\outputy(t_i)-\reduce{\outputy}(t_i)\|_2^2
				}
				{
					\sum_{i=0}^{\snapsize}
					\|\outputy(t_i)\|_2^2
				}
			}
			,
		\end{equation*}
		where $t_i=i\Delta t.$ 
		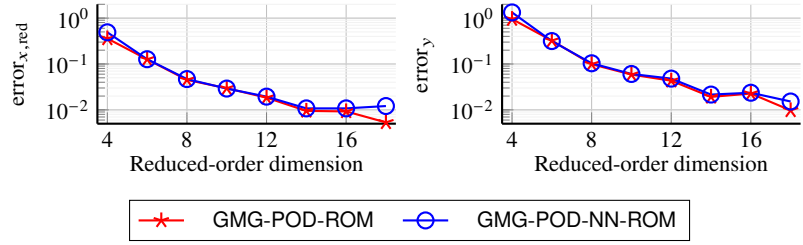
\begin{figure}[ht]
			\centering 
			\begin{subfigure}{0.45\linewidth}
				\begin{tikzpicture}
					\begin{semilogyaxis}[
						width=.75\linewidth,
						height=.30\linewidth, 
						scale only axis,
						grid=both,
						grid style={line width=.1pt, draw=gray!10},
						major grid style={line width=.2pt, draw=gray!50},
						axis lines*=left,
						axis line style={line width=\LineWidth},
						xmin=3.5, xmax=18.5,
						ymin=5*1e-3, ymax=2e0,
						xtick={4,8,12,16},    
						ytick={1e-3,1e-2,1e-1,1e0, 1e1},
						xlabel= Reduced-order dimension,
						ylabel= {$\errorstate$},
						xlabel style={at={(axis description cs:0.5,-0.2)},anchor=north},
						ylabel style={at={(axis description cs:-0.2,0.5)},anchor=south},
						]
						\addplot 
						[line width=\LineWidth, red, mark=star, mark size=\MarkSize]
						table[x index=0, y index=1, col sep=comma] {JR_error.dat};
						\addplot 
						[line width=\LineWidth, blue, mark=o, mark size=\MarkSize]
						table[x index=0, y index=1, col sep=comma] {NN_error.dat};
					\end{semilogyaxis}
				\end{tikzpicture}
			\end{subfigure}
			\begin{subfigure}{0.45\linewidth}
				\begin{tikzpicture}
					\begin{semilogyaxis}[
						width=.75\linewidth,
						height=.30\linewidth, 
						scale only axis,
						grid=both,
						grid style={line width=.1pt, draw=gray!10},
						major grid style={line width=.2pt, draw=gray!50},
						axis lines*=left,
						axis line style={line width=\LineWidth},
						xmin=3.5, xmax=18.5,
						ymin=5*1e-3, ymax=2e0,
						xtick={4,8,12,16},    
						ytick={1e-3,1e-2,1e-1,1e0, 1e1},
						xlabel= Reduced-order dimension,
						ylabel= {$\erroroutput$},
						xlabel style={at={(axis description cs:0.5,-0.2)},anchor=north},
						ylabel style={at={(axis description cs:-0.2,0.5)},anchor=south},
						legend to name=sharedlegend,
						legend columns=2,
						legend style=
						{font=\small, 
						column sep=1em},
						legend entries={\GMG-\POD-\ROM, \GMG-\POD-\NN-\ROM}
						]
						\addplot 
						[line width=\LineWidth, red, mark=star, mark size=\MarkSize]
						table[x index=0, y index=2, col sep=comma] {JR_error.dat};
						\addplot 
						[line width=\LineWidth, blue, mark=o, mark size=\MarkSize]
						table[x index=0, y index=2, col sep=comma] {NN_error.dat};
					\end{semilogyaxis}
				\end{tikzpicture}
			\end{subfigure}
			\vspace{0.1em}
			\begin{center}
				\pgfplotslegendfromname{sharedlegend}
			\end{center}
			\caption{Relative reduction and output error of the \GMG-\POD-\ROM and \GMG-\POD-\NN-\ROM. }
			\label{fig:MSD_comp_state_output_error}
		\end{figure}
		
		In Figure \ref{fig:MSD_comp_state_output_error}, we see that both the reduction error and output error of the \GMG-\POD-\ROM and \GMG-\POD-\NN-\ROM decrease as the reduced order increases. 
		For the reduced order above 6, the maximal deviation in the reduction and output errors are 
		\pgfplotstableread[col sep=comma]{error_div_state.dat}\DivErrorState
		\pgfplotstablegetelem{0}{0}\of\DivErrorState
		\pgfmathprintnumber[sci,precision=3]{\pgfplotsretval}\ 
		and
		\pgfplotstableread[col sep=comma]{error_div_output.dat}\DivErrorOutput
		\pgfplotstablegetelem{0}{[index]0}\of\DivErrorOutput
		\pgfmathprintnumber[sci,precision=3]{\pgfplotsretval}, respectively.	
		Thus, despite approximating $\reduce{\funcJ}(\stateRed)- \reduce{\funcR}(\stateRed)$ by neural networks, the performance of the \GMG-\POD-\NN-\ROM remains very close to that of the \GMG-\POD-\NN. 
			
		In the following,  we provide the simulation time of the \GMG-\POD-\ROM and the \GMG-\POD-\NN-\ROM of different sizes and compare them with the run time of the \FOM.\footnote{The evaluation of the gradient of the reduced-order Hamiltonian will also depend on the \FOM dimension $\FOMsize$. However, since this can be solved by using a structure-preserving \DEIM method \cite{chaturantabut2016structure}, we do not discuss it in this work.}
		The \FOM takes 
		\pgfplotstableread[col sep=comma]{FOM_run_time.dat}\runtimeFOM
		\pgfplotstablegetelem{0}{[index]1}\of\runtimeFOM
		\pgfmathprintnumber[fixed,precision=2,zerofill]\pgfplotsretval\,s, 
		while the \GMG-\POD-\ROM achieves a 
		\pgfplotstableread[col sep=comma]{speed_up_JR_FOM.dat}\speedUpJRFOM
		\pgfplotstablegetelem{0}{[index]0}\of\speedUpJRFOM
		\pgfmathprintnumber[fixed,precision=2,zerofill]\pgfplotsretval
		\,$\times$
		 faster speed-up with an average online simulation time over different reduced orders of 
		 \pgfplotstableread[col sep=comma]{run_time_POD_JR_ave.dat}\runtimeGMG
		 \pgfplotstablegetelem{0}{[index]0}\of\runtimeGMG
		 \pgfmathprintnumber[fixed,precision=2,zerofill]\pgfplotsretval\,s. The \GMG-\POD-\NN-\ROM improves this further, with an average simulation time of 
		 \pgfplotstableread[col sep=comma]{run_time_NN_ave.dat}\runtimeGMGNN
		 \pgfplotstablegetelem{0}{[index]0}\of\runtimeGMGNN
		 \pgfmathprintnumber[fixed,precision=2,zerofill]\pgfplotsretval\,s, 
		 which is a 
		 \pgfplotstableread[col sep=comma]{speed_up_NN_FOM.dat}\speedUpNNFOM
		 \pgfplotstablegetelem{0}{[index]0}\of\speedUpNNFOM
		 \pgfmathprintnumber[fixed,precision=2,zerofill]\pgfplotsretval
		 \,$\times$ speed-up compared to the \FOM. 
		In Fig. \ref{fig:NLMSD_comp_run_time}, we show an overview of the simulation time of the \ROMs with different reduced orders.
		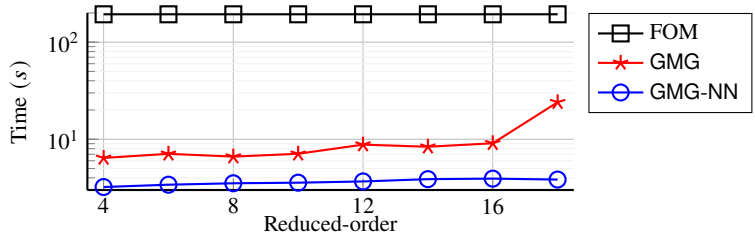
\begin{figure}[ht]
			\centering 
			\begin{tikzpicture}
				\begin{semilogyaxis}[
					width=.55\linewidth,
					height=.20\linewidth, 
					scale only axis,
					grid=both,
					grid style={line width=.1pt, draw=gray!10},
					major grid style={line width=.2pt, draw=gray!50},
					axis lines*=left,
					axis line style={line width=\LineWidth},
					xmin=3.5, xmax=18.5,
					ymin=3*1e-0, ymax=2*1e2,
					xtick={4,8,12,16},    
					ytick={1e0, 1e1,1e2, 1e3},
					legend columns=1,
					legend entries={FOM, \GMG, \GMG-\NN},
					legend style={font=\small},
					legend cell align={left},
					legend pos = outer north east,
					xlabel= {Reduced-order},
					ylabel= {Time $(s)$},
					xlabel style={at={(axis description cs:0.5,-0.1)},anchor=north},
					ylabel style={at={(axis description cs:-0.1,0.5)},anchor=south},
					]
					\addplot 
					[line width=\LineWidth, black, mark=square, mark size=\MarkSize]
					table[x index=0, y index=1, col sep=comma] {FOM_run_time.dat};
					\addplot 
					[line width=\LineWidth, red, mark=star, mark size=\MarkSize]
					table[x index=0, y index=1, col sep=comma] {run_time_POD_JR.dat};
					\addplot 
					[line width=\LineWidth, blue, mark=o, mark size=\MarkSize]
					table[x index=0, y index=1, col sep=comma] {run_time_NN.dat};
				\end{semilogyaxis}
			\end{tikzpicture}
			\caption{Comparison of online simulation time of the \FOM and \ROMs}
			\label{fig:NLMSD_comp_run_time}
		\end{figure}
	
	\section{Conclusion and outlook}
	\label{Sec:Conclusion_Outlook}
	In this work, we present a two-step \MOR method which results to \GMG-\POD-\NN-\ROM. 
	First, a \GMG-\POD-\ROM is constructed,
	and second, 
	the interconnection matrix and dissipation matrix are approximated via a neural network.
	Moreover, the \pH structure of the \GMG-\POD-\ROM is preserved due to the design of the \NN.
	
	In the numerical example, we highlight that the \GMG-\POD-\NN-\ROM performs closely to 
	the \GMG-\POD-\ROM in terms of accuracy.
	Moreover, \GMG-\POD-\NN-\ROM achieves 
	 \pgfplotstableread[col sep=comma]{speed_up_NN_JR.dat}\speedUpNNJR
	\pgfplotstablegetelem{0}{[index]0}\of\speedUpNNJR
	\pgfmathprintnumber[fixed,precision=2,zerofill]\pgfplotsretval
	\,$\times$ and 
	\pgfplotstablegetelem{0}{[index]0}\of\speedUpNNFOM
	\pgfmathprintnumber[fixed,precision=2,zerofill]\pgfplotsretval
	\,$\times$ 
	speed-up compared to the \GMG-\POD-\ROM and the \FOM, respectively.

	Here, we did not consider the case of combining the \GMG reduction with a nonlinear approximation map, which is a subject for future work. 
	\begin{acknowledgement}
		SG and HMon acknowledge support from the DEPMAT project (NWO project number N21022h).
	\end{acknowledgement}
	\ethics{Competing Interests}{
		The authors have no conflicts of interest to declare that are relevant to the content of this chapter.}
	
	\bibliographystyle{spmpsci} 
	\bibliography{References.bib} 

@article{mehrmann2023control,
	title={Control of port-{H}amiltonian differential-algebraic systems and applications},
	author={Mehrmann, Volker and Unger, Benjamin},
	journal={Acta Numerica},
	volume={32},
	pages={395--515},
	year={2023},
	publisher={Cambridge University Press},
	doi={10.1017/S0962492922000083}
}

@article{kotyczka2018discrete,
	title={Discrete-time port-{H}amiltonian systems based on {G}auss-{L}egendre collocation},
	author={Kotyczka, Paul and Lefevre, Laurent},
	journal={IFAC-PapersOnLine},
	volume={51},
	number={3},
	pages={125--130},
	year={2018},
	publisher={Elsevier},
	doi={10.1016/j.ifacol.2018.06.035}
}

@article{schwerdtner2022structure,
	title={Structure-preserving model order reduction for index one port-{H}amiltonian descriptor systems},
	author={Schwerdtner, Paul and Moser, Tim and Mehrmann, Volker and Voigt, Matthias},
	journal={arXiv preprint arXiv:2206.01608},
	year={2022}
}

@article{GM2026structure,
	title={Structure-preserving model reduction on manifolds of port-{H}amiltonian systems},
	author={Glas, Silke and Mu, Hongliang},
	journal={arXiv preprint arXiv:2603.08656},
	year={2026}
}

@article{rettberg2025data,
	title={Data-driven identification of latent port-{H}amiltonian systems},
	author={Rettberg, Johannes and Kneifl, Jonas and Herb, Julius and Buchfink, Patrick and Fehr, J{\"o}rg and Haasdonk, Bernard},
	journal={Computational Science and Engineering},
	volume={2},
	number={1},
	pages={4},
	year={2025},
	publisher={Springer},
	doi = {10.1007/s44207-025-00007-2}
}

@article{shanker1996effect,
	title = {Effect of data standardization on neural network training},
	journal = {Omega},
	volume = {24},
	number = {4},
	pages = {385-397},
	year = {1996},
	issn = {0305-0483},
	doi = {10.1016/0305-0483(96)00010-2},
	author = {M. Shanker and M.Y. Hu and M.S. Hung},
}

@article{bonomi2017matrix,
	author = {Diana Bonomi and Andrea Manzoni and Alfio Quarteroni},
	doi = {10.1016/j.cma.2017.06.011},
	issn = {00457825},
	journal = {Computer Methods in Applied Mechanics and Engineering},
	month = {9},
	pages = {300-326},
	publisher = {Elsevier B.V.},
	title = {A matrix {DEIM} technique for model reduction of nonlinear parametrized problems in cardiac mechanics},
	volume = {324},
	year = {2017}
}

@article{califano2021geometric,
	title={Geometric and energy-aware decomposition of the {N}avier--{S}tokes equations: {A} port-{H}amiltonian approach},
	author={Califano, Federico and Rashad, Ramy and Schuller, Frederic P and Stramigioli, Stefano},
	journal={Physics of fluids},
	volume={33},
	number={4},
	year={2021},
	publisher={AIP Publishing},
	doi = {10.1063/5.0048359}
}

@article{falaize2016passive,
	title={Passive guaranteed simulation of analog audio circuits: A port-{H}amiltonian approach},
	author={Falaize, Antoine and H{\'e}lie, Thomas},
	journal={Applied Sciences},
	volume={6},
	number={10},
	pages={273},
	year={2016},
	publisher={MDPI},
	doi = {10.3390/app6100273}
}

@book{Duindam2009,
	author = {Vincent Duindam and Alessandro Macchelli and Stefano Stramigioli and Herman Bruyninckx},
	doi = {10.1007/978-3-642-03196-0},
	isbn = {978-3-642-03195-3},
	title = {Modeling and Control of Complex Physical Systems The Port-Hamiltonian Approach Springer},
	year = {2009},
	publisher = {Springer}
}

@article{schulze2023structure,
	author = {Philipp Schulze},
	journal = {Frontiers in Applied Mathematics and Statistics},
	pages = {1160250},
	title = {Structure-preserving model reduction for port-{H}amiltonian systems based on separable nonlinear approximation ansatzes},
	volume = {9},
	year = {2023},
	doi = {10.3389/fams.2023.1160250}
}

@article{vanderschaft2014port,
	author = {Arjan J. van der Schaft and Dimitri Jeltsema},
	doi = {10.1561/2600000002},
	issn = {23256826},
	issue = {2-3},
	journal = {Foundations and Trends in Systems and Control},
	pages = {173-378},
	publisher = {Now Publishers Inc},
	title = {Port-{H}amiltonian systems theory: {A}n introductory overview},
	volume = {1},
	year = {2014}
}

@article{chaturantabut2016structure,
	author = {Saifon Chaturantabut and Christopher Beattie and Serkan Gugercin},
	doi = {10.1137/15M1055085},
	issn = {10957197},
	issue = {5},
	journal = {SIAM Journal on Scientific Computing},
	pages = {B837-B865},
	publisher = {Society for Industrial and Applied Mathematics Publications},
	title = {Structure-preserving model reduction for nonlinear port-{H}amiltonian systems},
	volume = {38},
	year = {2016}
}

@article{polyuga2012effort,
	author = {Rostyslav V. Polyuga and Arjan J. van der Schaft},
	doi = {10.1016/j.sysconle.2011.12.008},
	issn = {01676911},
	issue = {3},
	journal = {Systems and Control Letters},
	pages = {412-421},
	title = {Effort- and flow-constraint reduction methods for structure preserving model reduction of port-{H}amiltonian systems},
	volume = {61},
	year = {2012}
}

@article{gugercin2012structure,
	author = {Serkan Gugercin and Rostyslav V. Polyuga and Christopher Beattie and Arjan J. van der Schaft},
	doi = {10.1016/j.automatica.2012.05.052},
	issn = {00051098},
	issue = {9},
	journal = {Automatica},
	pages = {1963-1974},
	title = {Structure-preserving tangential interpolation for model reduction of port-{H}amiltonian systems},
	volume = {48},
	year = {2012}
}

@article{kawano2018structure,
	author = {Yu Kawano and Jacquelien M.A. Scherpen},
	doi = {10.1109/TAC.2018.2811787},
	issn = {15582523},
	issue = {12},
	journal = {IEEE Transactions on Automatic Control},
	pages = {4286-4293},
	publisher = {Institute of Electrical and Electronics Engineers Inc.},
	title = {Structure preserving truncation of nonlinear port {H}amiltonian systems},
	volume = {63},
	year = {2018}
}

@article{buchfink2024model,
	author = {Patrick Buchfink and Silke Glas and Bernard Haasdonk and Benjamin Unger},
	doi = {10.1016/j.physd.2024.134299},
	issn = {01672789},
	journal = {Physica D: Nonlinear Phenomena},
	publisher = {Elsevier B.V.},
	title = {Model reduction on manifolds: {A} differential geometric framework},
	volume = {468},
	year = {2024}
}
\end{document}